\documentclass{amsart}
\author{Umar Hayat}
\title{The Cramer Varieties $Cr(r,r+s,s)$}
\date{Wednesday, May 29, 2013}
\address{Umar Hayat, School of Mathematical Sciences, University College Dublin, Belfield, Dublin 4, Ireland}
\address{Department of Mathematics, Quaid-iAzam University, Islamabad, Pakistan}
\email{umarmaths@gmail.com}
\usepackage{amsmath,bm,graphicx}
\usepackage{subfigure}
\usepackage{amssymb}
\usepackage{amsfonts}
\usepackage{amsthm}
\newcommand{\GL}{GL}
\newcommand{\diago}{diagonal}

\newcommand{\Sym}{S}

\newcommand{\Hom}{Hom}
\newcommand{\Gr}{Gr}
\newcommand{\Stab}{Stab}

\newcommand{\Cr}{Cr}
\newcommand{\OGr}{OGr}

\newcommand{\IC}{\mathbb{C}}

\newcommand{\IT}{\mathbb{T}}

\newcommand{\cO}{\mathcal{O}}

\newtheorem{pro}{Proposition}
\newtheorem{thm}{Theorem}

\begin{document}
\begin{abstract}
In this paper we study quasi-homogeneous affine
algebraic varieties,  that is,  varieties obtained
as closures of orbits of suitable group representations. We also discuss one interesting case that has links with the Orthogonal Grassmannian $\OGr(5,10)$. The main aim is to write the tangent bundle and the canonical class of quasi-homogeneous affine algebraic varieties in terms of group representations.

\end{abstract}



\maketitle

\pagestyle{myheadings}

\markboth{UMAR HAYAT}{The Cramer Varieties $Cr(r,r+s,s)$}

\section{Introduction}
\label{sec-intro}
A homogeneous space for an algebraic group $G$ is a space $M$ with a transitive action of $G$ on $M$. Equivalently, it is a space of the form $G/H$, where $G$ is an algebraic group and $H$ a closed subgroup of $G$. Homogeneous spaces play a vital role in the representation theory of the algebraic group because representations are often realised as the space of sections of vector bundles over homogeneous spaces. Homogeneous spaces have been studied in different context but not much is known about quasi-homogeneous spaces.

Reid and Corti \cite{RC} studied weighted analogs of the homogeneous spaces such as Grassmannian $\Gr(2,5)$ and the Orthogonal Grassmanian $\OGr(5,10)$ and how to use these as weighted projective  constructions. 

In this paper, we study quasi-homogeneous spaces arising from two sets of matrix equations. Let $M$ and $N$ be $r\times t$ and $t\times s$ matrices with $t=r+s$, $r\leq s$. We define a variety $V \subset \IC^{(r\times s)+(s\times t)+1}$ by the two sets of equations
$$
MN=0 \qquad \text{ and } \qquad \omega \bigwedge^{r}M=\bigwedge^{s}N,
$$
where $\omega \in \IC $ and $ {\bigwedge } ^{r} M$, $\bigwedge^{s}N$ denote the $r\times r$ and $s\times s$ minors of $M$ and $N$ respectively. We follow the method given in \cite{HOC} to equate $\omega$ times the $r\times r$ minors of $M$ with the $s\times s$ minors of $N$. We denote these varieties by $\Cr(r,r+s,s)$ and call them Cramer varieties.

The case $r=3$ and $s=1$ is the original codimension $4$ example in the first paper of Kustin and Miller \cite{KM}. Understanding this case led them to the more general notion of Gorenstein unprojection. For $r=4$ and $s=1$, these varieties in terms of  equations are unprojections and have been studied by Papadakis and Reid, see \cite{PR, PR1}. In fact for $s=1$ the variety is a single unprojection because all $x_{i}=0$ is a codimension $(r+1)$ complete intersection $D$ and all $\sum m_{ij}x_{j}=0$ is a codimension $r$ complete intersection $X$ containing $D$. So Kustin-Miller unprojection applies to give $\omega$ as an unprojection variable with the second set of equations as unprojection equations. Also Hochster \cite{HOC} studied these examples in relation to the variety of complexes.\\

We prove that the canonical divisor of the  Cramer variety $\Cr(r,t,s)$ is Cartier in Proposition $1$. Our main result is Theorem $1$ in section $3$ that gives criteria to calculate the canonical class of the $\Cr(r,t,s)$ in terms of weights. 

\section{The variety $\Cr(r,r+s,s)$ in equations}
 Let $M$ and $N$ be the $r\times t$ and $t\times s$ matrices as given below:

\[
 M =
 \begin{pmatrix}
  m_{11} & m_{12} & \cdots & m_{1t} \\
  \vdots  & \vdots  & \ddots  & \vdots \\ 
  m_{r1} & m_{r2} & \cdots & m_{rt} \\
   \end{pmatrix}
   \qquad
\text{and }
\qquad
 N =
 \begin{pmatrix}
  n_{11} & n_{12} & \cdots  & n_{1s} \\
   \vdots  & \vdots  & \ddots& \vdots    \\
  n_{t1} & n_{t2} & \cdots & n_{ts}\\
 \end{pmatrix}.
\]

We define a variety $V \subset \IC^{(r\times t)+(t\times s)+1}$ by the two sets of equations
$$
MN=0 \qquad \text{ and } \qquad \omega \bigwedge^{r}M=\bigwedge^{t}N,
$$
where $\omega\in \IC$ and $\bigwedge^{r}M$, $\bigwedge^{s}N$ denote the $r\times r$ and $s\times s$ minors of $M$ and $N$ respectively. We follow the method given in \cite{HOC} to equate $\omega$ times the $r \times r$ minors of $M$ with the $s\times s$ minors of $N$. Let $T=\{1, \dots, t\}$ be a set of $t$ elements and $T_{r}$ be a subset of any $r$ elements of $T$. Let $M_{1, \dots, r}$ be the minor of $M$ obtained from any $r$ columns of $M$ and let $N_{\widehat{1, \dots, r}}$ be the minor of $N$ obtained by deleting the $r$ rows of $N$. We equate these minors as follows
$$
(-1)^{\Sigma} \omega M_{1, \dots, r}= N_{\widehat{1, \dots, r}},
$$
where $\Sigma= \sum t_{i}$.
Now if we consider 
$$
V_{0}= \bigg\{( M, N, \omega ) : \text{rank of } M=  r, \text{rank of } N= s  \text{  and } \omega\neq 0 \bigg\}, 
$$ 
then  $V_{0}$ has codimension $ts+1-s^{2}$ in $\IC^{(r\times t)+(t\times s)+1}$ and $V=\overline{V_{0}}$. The $V_{0}$ is a homogeneous space, the orbit of the vector  
\[
\left(
 M_{0} =
 \begin{pmatrix}
  I_{r\times r} & 0_{r\times s}  \\
     \end{pmatrix}
\text{, }
\quad
 N_{0} =
 \begin{pmatrix}
  0_{r\times s} \\
  I_{s\times s}  \\
   \end{pmatrix}
   \quad
   \text{ and } 
   \omega=1
   \right).
\]
under the action of $G=\GL(r)\times \GL(t)\times \GL(s)$ as explained in section \ref{s!Qausi}.

When $M$ and $N$ are of  maximal rank then we can assume the first minor $M_{1\cdots r}$ of $M$ is nonzero. We can use that to solve the top $r$ rows of $N$ and $\omega$ in terms of remaining entries of $N$ and $M$.

   Let $S=V\setminus V_{0}$ be the complement of $V_{0}$ in $V$. If $(M, N, \omega)\in S$ then for $\bigwedge^{s}N=0$ we have three possibilities for the elements of $S$:

\begin{enumerate}
  \item either the rank of $M$ is full and $\omega=0$; 
  \item or the rank of $M$ is strictly less than r and $\omega\neq0$;
  \item or the rank of $M$ is strictly less than r and $\omega=0$. 
\end{enumerate}
In case $1$, when $M$ is of maximal rank then one of the minors of $M$ is nonzero. For $N$ of rank $s-1$, we get a codimension one irreducible variety, say $V_{1}$. A typical element of a divisor $V_{1}$ looks like 
\[
\left(
 M_{0} =
 \begin{pmatrix}
  I_{r\times r} & 0_{r\times s}  \\
     \end{pmatrix}
\text{, }
\quad
 N_{0} =
 \begin{pmatrix}
  0_{r+1\times s} \\
  I_{s-1\times s-1}  \\
   \end{pmatrix}
   \quad
   \text{ and } 
   \omega=0
   \right).
\]
 
In all other cases where rank $N\leq s-1$ or rank $M\leq r-1$  the codimension is greater than or equal to two so we are not worried about that part of $S$: these subvarieties are not divisorial so do not appear in the canonical class.
\subsection{The canonical class of $\Cr(r,r+s,s)$}\label{s!KV}
 Suppose that the first minor $M_{1, \cdots ,r}$ of $M$ is  nonzero. Then we can write entries of first $r$ rows of $N$, $n_{11}$,$\cdots$  ,$n_{1s}$, $n_{21}$, $n_{22}$,$\cdots$ ,$n_{2s}$, $\cdots$ $n_{r1}$, $n_{r2}$,$\cdots$ ,$n_{rs}$  and $\omega$ in terms of remaining entries of $M$ and $N$. 
 Similarly if we assume that the minor $M_{2,\cdots ,r+1}$ of $M$ is nonzero then we can solve for $n_{21}$,$\cdots$  ,$n_{2s}$, $n_{31}$, $n_{32}$,$\cdots$ ,$n_{3s}$, $\cdots$ ,$n_{r+1,1}$, $n_{r+1,2}$,$\cdots$ ,$n_{r+1,s}$  and $\omega$. In the same way if we assume that any $r\times r$  minor  of $M$ is nonzero then we can use that to solve for the $r$ rows of $N$ and $\omega$ where the coordinates will be remaining entries of $M$ and $N$.
 
Let $U_{M_{1,\cdots , r}\neq 0}$ and $U_{M_{2, \cdots , r+1}\neq 0}$  be the two charts for $V$ with coordinates \linebreak
$\xi_{1},\dots, \xi_{rt}, \xi_{rt+1} \cdots , \xi_{rt+s^{2}}$ and $\eta_{1},\dots, \eta_{rt}, \eta_{rt+1}, \eta_{rt+s^{2}}.$
There are $rt+s^{2}-s$ coordinates common to both charts because these charts only differ by one row of $N$. The change of coordinates from one chart to other is given by the $rt+s^{2}$ Jacobian matrix $J$ whose first $rt+s^{2}-s$ block is $I_{rt+s^{2}-s\times rt+s^{2}-s}$:
\[
J=
 \begin{pmatrix} 
 I_{rt+s^{2}-s\times rt+s^{2}-s} & 0_{rt+s^{2}-s \times s}\\
 0_{s \times rt+s^{2}-s} & C\\
  \end{pmatrix},  
  \] 
where $C$ is the $s\times s$ diagonal matrix whose diagonal entries are $\dfrac{ M_{2, \cdots, r+1}}{ M_{1,\cdots, r}}$   and the determinant of $J$ is $\bigg(\dfrac{ M_{2,\cdots, r+1}}{ M_{1,\cdots, r}}\bigg) ^{s}$.

The sheaf of the canonical differentials is 
 $$
\cO(K_{V})=\bigwedge^{rt+s^{2}}\Omega^{1}_{V} \text{ and }
\cO(K_{V})\mid_{U_{M_{1,\cdots, r}\neq 0}}=\cO_{U_{M_{1, \cdots, r}\neq 0}}\cdot \sigma_{1, \cdots, r} 
$$
where
\begin{equation}\label{s12}
 \sigma_{1, \cdots, r}= \dfrac{d\xi_{1}\wedge \dots \wedge d\xi_{rt+s^{2}}}{(M_{1,\cdots, r})^{s}} \text{ and similarly }  \sigma_{2, \cdots, r+1}= \dfrac{d\eta_{1}\wedge \dots \wedge d\eta_{rt+s^{2}}}{(M_{2, \cdots, r+1})^{s}}.
 \end{equation}
The minor $M_{1,\cdots, r}$ is invertible on $M_{1, \cdots ,r}$ and putting $(M_{1,\cdots, r})^{s}$ in the denominator is a convenient trick to cancel out the Jacobian matrix, which will appear again later. The above calculation of the Jacobian determinant shows that $\sigma_{1,\cdots, r}= \sigma_{2,\cdots, r+1}$ and repeating the same calculation defines $\sigma= \sigma_{i_{1}, \cdots, i_{r}}$  independently of $i_{1}, \cdots, i_{r}$. Since $\sigma_{i_{1}, \cdots, i_{r}}$ is a basis of $\bigwedge^{rt+s^{2}}\Omega^{1}_{V}$ and has no zeros or poles, exactly because of the $M_{i_{1}, \cdots, i_{r}}$ in the denominator, we have   $$
K_{V}=divisor(\sigma)=0.
$$
\begin{pro}
{\it The canonical divisor of the  Cramer variety $\Cr(r,t,s)$ is Cartier}.

\end{pro}

\begin{proof}
A Cartier divisor $K$ on a variety is an open cover $\{(U_{i})\}$ and rational functions $f_{i} \in k(U_{i})^{\ast}$ such that for all $i$, $j$, $f_{i} f_{j}^{-1} \in \cO^{\ast}(U_{i}\cap U_{j})$. We have an open cover  $\{ (U_{M_{i_{1}, \cdots, i_{r}}}) \}$ for $\Cr(r,t,s)$, with transition functions $\dfrac{1}{M_{i_{1}, \cdots, i_{r}}} \in k(U_{M_{i_{1}, \cdots, i_{r}} })^{\ast}$. Note that  $\dfrac{M_{1,\cdots, r}^{s}}{M_{2, \cdots , r+1}^{s}} \in \cO^{*}(U_{M_{1, \cdots, r} }\cap U_{M_{2, \cdots, r+1} })= \cO^{*}(U_{M_{1, \cdots, r} },_{M_{2, \cdots,r+1 } })$. Hence $K$ is Cartier for $\Cr(r,t,s)$. 
\end{proof}
\section{The variety $\Cr(r,t,s)$ as a quasi-homogeneous space}\label{s!Qausi}
Our aim is to study the variety $V$ as the closure of the orbit of a special vector. Let $G=\GL(r)\times \GL(t)\times \GL(s)$ which is a reductive algebraic group. Let $W_{r}$, $W_{t}$ and $W_{s}$ be the given $r$, $t$ and $s$-dimensional representations of $\GL(r)$, $\GL(t)$ and $\GL(s)$ respectively.

We want to define an action of $G=\GL(r)\times \GL(t)\times \GL(s)$ on the representation $R=\Hom(W_{r},W_{t})\oplus\Hom(W_{t}, W_{s})\oplus \IC$ such that the variety $V$ is invariant under this action. In coordinate-free terms, $M\in\Hom(W_{r},W_{t})$, $N\in\Hom(W_{t}, W_{s})$ and $\omega\in\IC$ and the action of $(A,B,C)\in G$ with $A\in\GL(r)$, $B\in\GL(t)$, $C\in\GL(s)$ is defined  as follows,
\begin{align*}
M &\longmapsto AMB^{-1} \\
N &\longmapsto BNC^{-1} \\
\omega &\longmapsto \lambda \omega,  \text{ where } \lambda =\dfrac{\det(B)}{\det(A)\times \det(C)}.
\end{align*}
Let $M$ and $N$ be matrices of maximal rank. By using row and column operations we can write $M$ and $N$ as follows
\[
\left(
 M_{0} =
 \begin{pmatrix}
  I_{r\times r} & 0_{r\times s}  \\
     \end{pmatrix}
\text{, }
\quad
 N_{0} =
 \begin{pmatrix}
  0_{r\times s} \\
  I_{s\times s}  \\
   \end{pmatrix}
   \quad
   \right).
\]
The stabiliser of $v=\left(  M_{0},\text{ } N_{0}, \text{ } \omega=1 \right) $ is 
$$
H=\Stab(v)=\left\{\left( A,B, C\right) \text{ }\bigg \vert 
 B =
 \begin{pmatrix}
 A & 0 \\ 
 * & C\\
 \end{pmatrix}
 \right\}, 
$$
where ($*$) means there is no restriction on this block. One can observe that $v$ is not the highest weight vector, it is not even a weight vector. In this part we consider the open orbit 
$$
V_{0}=G/H\simeq G \cdot v \hookrightarrow R.
$$
Cramer variety $\Cr(r,t,s)$ is a quasi-homogeneous space with the natural action of $G$ and 
$V=\overline{G \cdot v}.
$
\subsection{The Weyl group $W(G)$}
To study the algebraic group $G=\GL(r)\times \GL(t)\times \GL(s)$ and its representations we use the Weyl group  $W(G)\cong \Sym_{r}\times \Sym_{t}\times \Sym_{s}$ which acts as a permutation group. We know from section \ref{s!Qausi}  that  $R=\Hom(W_{r},W_{t})\oplus\Hom(W_{t}, W_{s})\oplus \IC$ is a representation of $G$. We describe here how the Weyl group  acts on $R$. The group $\Sym_{r}$ acts on any $M\in \Hom(W_{r},W_{t})$ from the left and permutes the rows while $\Sym_{t}$ acts on the right and permutes the columns.
Similarly, $\Sym_{t}$ acts on $N\in \Hom(W_{t}, W_{s})$ from the left and permutes the rows and $\Sym_{s}$ acts on right and permutes the columns.
\subsection{The torus action and Weyl group}
If $\IT\subset G$ is the maximal torus given by
\begin{multline*}
\Biggl[  T_{A}=
\begin{pmatrix}
  a_{11} & 0 & \cdots & 0 \\
  0 & a_{22} & \cdots & 0 \\
  \vdots  & \vdots  & \ddots & \vdots \\
  0 & 0 & \cdots & a_{rr} \\
  \end{pmatrix}, \text{  }
T_{B}= 
 \begin{pmatrix}
  b_{11} & 0 & \cdots & 0 \\
  0 & b_{22} & \cdots & 0 \\
  \vdots  & \vdots  & \ddots & \vdots \\
  0 & 0 & \cdots & b_{tt} \\
  \end{pmatrix},
  \text{  }\\
      T_{C}= 
 \begin{pmatrix}
  c_{11} & 0 & \cdots & 0 \\
  0 & c_{22} & \cdots & 0 \\
  \vdots  & \vdots  & \ddots & \vdots \\
  0 & 0 & \cdots & c_{ss} \\
  \end{pmatrix} 
 \Biggr] ,
\end{multline*}  
 then it acts on 
 \[
 M_{11} =
 \begin{pmatrix}
  m_{11} & 0_{1\times t-1} \\
  0_{r-1\times 1} &  0_{r-1\times s-1}\\
   \end{pmatrix}
   \quad
\text{and }
\quad
 N_{11} =
 \begin{pmatrix}
  n_{11} & 0_{1\times s-1} \\
  0_{t-1\times 1} & 0_{t-1\times s-1}\\
 
 \end{pmatrix}
 \quad
 \text{ and } \omega
\] 
as explained in section $3$. Under this action $M_{11}$, $N_{11}$ and $\omega$ are the weight vectors with weights $\frac{a_{11}}{b_{11}}$, $\frac{b_{11}}{c_{11}}$ and $\det(T_{B})\times (\det T_{A})^{-1}\times (\det T_{C})^{-1}$ respectively.

\subsection{One parameter subgroup and elements of $V \setminus V_{0}$} 
 Let \linebreak $P(t)= \left(  T_{A}(t), T_{B}(t) \text{ and } T_{C}(t) \right) $ be the one parameter subgroup given by 
 \begin{multline*}
\Biggl[  T_{A}(t)=I_{r\times r}
, \text{  }
T_{B}(t)= \diago(T_{A})(1,...,1,t_{r+1,r+1},1,...,1)
 ,
  \text{}
 T_{C}(t)= I_{s\times s}  
 \Biggr].
\end{multline*}

In section $2$ we showed that we have only one irreducible divisor $V\backslash V_{0}$, say $V_{1}$. We want to show that $V_{1}\subset \overline{G \cdot v}$. For this we show that there exists a one parameter subgroup $P(t)$  such that $V_{1} \subset \overline{P(t)\cdot v}$. \\
In fact the one parameter subgroup $P(t)= \left(  T_{A}(t), T_{B}(t) \text{ and } T_{C}(t) \right) $  acts on 
\[
\left(
 M_{0} =
 \begin{pmatrix}
  I_{r\times r} & 0_{r\times s}  \\
     \end{pmatrix}
\text{, }
\quad
 N_{0} =
 \begin{pmatrix}
  0_{r\times s} \\
  I_{s\times s}  \\
   \end{pmatrix}
   \quad
   \text{ and } 
   \omega=1
   \right).
\]

 and we get 
\[
\left(
 M_{0} =
 \begin{pmatrix}
  I_{r\times r} & 0_{r\times s}  \\
     \end{pmatrix}
\text{, }
\quad
 N_{0} =
 \begin{pmatrix}
  0_{r\times s} \\
  t_{r+1,r+1} & 0_{1\times s-1} \\
  I_{s-1\times s}  \\
   \end{pmatrix}
   \quad
   \text{ and } 
   \omega=t
   \right).
\]
 
Therefore we have shown that there exists one parameter subgroup $P(t)$ such that $V_{1}\subset\overline{P(t) \cdot v}$. That is for $t\rightarrow 0$ along the one parameter subgroup we get the typical vector of $V_{1}$. 
\subsection{The relationship between the canonical class of $\Cr(r,t,s)$ and representations of $G$}

In this section we observe the relationship between the tangent bundle and  the canonical class to the open orbit and the representations of $G$. We know from  section $2$ that the tangent bundle $T_{V}$ to $V$ is a vector bundle of rank $rt+s^{2}$ and  top wedge of its dual is equal to 
$$
\bigwedge^{rt+s^{2}}T^{\vee}_{V}=\bigwedge^{rt}\Hom(W_{r},W_{t})\otimes\bigwedge^{s^{2}}\Hom(W^{\prime}_{s}\subset W_{t},W_{s})
$$
where $W_{r}$, $W_{s}$, $W^{\prime}_{s}$ and $W_{t}$ are the given $r$, $s$, $s$ and $t$-dimensional $\GL(r)$-, $\GL(s)$-, $\GL(s)$- and $\GL(t)$- representations. 
\subsection{The canonical class in terms of weights}
The canonical differential  $\sigma_{1, \cdots, r}$  in equation (\ref{s12}) is a weight vector for the maximal $\IT \subset G$ with weight \linebreak 
$\dfrac{(\det(A))^{r}\times \det(B)}{(\det(C))^{s}}$ and similarly all the $\sigma= \sigma_{i_{1}, \cdots, i_{r}}$ are weight vectors with  weight $\dfrac{(\det(A))^{r}\times \det(B)}{(\det(C))^{s}}$. The only problem is $\IT \subset G$ does not normalise the stabiliser $H$. If we take the restricted torus $\IT_{H}=\IT \cap N_{H}$, where $N_{H}$ is the normaliser of $H$ then $\IT_{H}$ 
$$
\left\{  T_{A}, \text{  }
T_{B}= 
 \begin{pmatrix}
  T_{A} &  0_{r\times s} \\
  0_{s\times r}  & T_{C} \\
    \end{pmatrix},
  \text{  }  
  T_{C}  
 \right\} 
 $$ 
 acts on the canonical differential  $\sigma_{1, \cdots, r}$  in equation (1) and $\sigma_{1, \cdots, r}$ is a weight vector with weight 
 $\dfrac{(\det\IT_{A})^{s}}{(\det\IT_{C})^{r}}$.
 
 Let $\mathfrak{h}$ and $\mathfrak{g}$ be the Lie algebras of $H$ and $G$ respectively. The tangent bundle $T_{G/H}$ to $G/H$ comes from the representation $\mathfrak{g/h}$: $\mathfrak{g/h}$ is the tangent space to $G/H$ at the identity $H$, and the tangent space to any other $gH\in G/H$ is given by $\mathfrak{g}/g\mathfrak{h}g^{-1}$. The canonical class of the variety $G/H$ is  
\[
K_{G/H}= divisor(\overset{rt+s^{2}}{\bigwedge} T^{\vee}_{G/H}). 
\]
The canonical differential $K_{G/H}$ is a weight vector for $\IT_{H}$ and its weight is exactly the product of those weights of $G$ which are not weights of $H$. 
\begin{thm}
{\it The weight of the canonical differential $K_{G/H}$ is the determinant of the restricted torus.}
\end{thm}

\begin{proof}
In this case of the Cramer variety $\Cr(r,t,s)$, there are $rt+s^{2}$ weights that are weights of $G$ but not of $H$ and the product of those weight is $\dfrac{(\det\IT_{A})^{s}}{(\det\IT_{C})^{r}}$. 

In the matrix below, each asterisk block represents a collection of  weight spaces of $G$ that are not the weight spaces for $H$,
\[
\left(\begin{array}{c|c}
\ast_{r\times r} &  \ast_{r\times s} \\ \hline
0 &\ast_{s\times s}   \\ 
     
	\end{array}\right).
\]
When we multiply these weights then the only contribution comes from the top right $r\times s$ block because the product of all the  weights in the $r\times r$ and $s\times s$ square blocks is $1$.

This shows that the canonical differential is a multiple of the  determinant $\dfrac{(\det\IT_{A})^{s}}{(\det\IT_{C})^{r}}$ under the action of $\IT_{H}$ on $\mathfrak{g/h}$.
\end{proof}
\subsection{Special case $\Cr(2,4,2)= \OGr(5,10)$}
Let $M$ and $N$ be $2\times 4$ and $4\times 2$ matrices respectively given by 
\[
 M =
 \begin{pmatrix}
  m_{11} & m_{12} & m_{13} & m_{14} \\
  m_{21} & m_{22} & m_{23} & m_{24} \\
   \end{pmatrix}
   \qquad
\text{and }
\qquad
 N =
 \begin{pmatrix}
  n_{11} & n_{12}  \\
n_{21} & n_{22}  \\
n_{31} & n_{32}  \\
  n_{41} & n_{42} \\
 \end{pmatrix}.
\]
We define a variety $V \subset \IC^{(2\times 4)+(4\times 2)}$ by the two sets of equations
$$
MN=0 \qquad \text{ and } \qquad  \bigwedge^{2}M=\bigwedge^{2}N.
$$

If $M$ is not of maximal rank then it follows from the second equation that $N$ must be of rank less than $2$. This is locus of codimension greater than or equal to $2$ so we are not worried about this locus when calculating the canonical class of the variety. 

We can assume the first entry $m_{11}$ of $M$ is nonzero. We can use that to solve the top row of $N$ and $m_{22}$, $m_{23}$ and $m_{24}$ in terms of remaining entries of $M$ and $N$. Similarly if we assume the entry $m_{21}$ of $M$  is nonzero then we can use that to eliminate the first row of $N$ and $m_{12}$, $m_{13}$ and $m_{14}$ in terms of remaining entries of $M$ and $N$.

Let $U_{m_{11}}$ and $U_{m_{21}}$  be the two charts for $V$ with coordinates given above. These two charts differ by three coordinates and the Jacobian determinant is given by $\bigg(\dfrac{m_{11}}{m_{21}}\bigg)^{3}$.

We know that
 $$
\cO(K_{V})=\bigwedge^{11}\Omega^{1}_{V} \text{ and }
\cO(K_{V})\mid_{U_{m_{11}\neq 0}}=\cO_{U_{m_{11}\neq 0}}\cdot \sigma_{11} 
$$
where $\sigma_{11}=\big( dm_{11}\wedge \dots \wedge dm_{14} \wedge dm_{21} \wedge dn_{21}  \wedge dn_{22} \wedge dn_{31} \wedge dn_{32} \wedge dn_{41} \wedge dn_{42} \big)/(m_{11})^{3}$.  Similarly 
$\sigma_{12}=\big( dm_{21}\wedge \dots \wedge dm_{24} \wedge dm_{11} \wedge dn_{21}  \wedge dn_{22} \wedge dn_{31} \wedge dn_{32} \wedge dn_{41} \wedge dn_{42} \big)/(m_{21})^{3}$, with $\sigma_{11}=\sigma_{12}$ and repeating the same calculation gives that $\sigma= \sigma_{ij}$ is independent of $ij$. Since $\sigma_{ij}$ is a basis for $\Omega_{V}^{11}$ on $V_{m_{ij}}$ (no zeros or poles, exactly because of the $m_{ij}$ in the denominator), we have
$$
K_{V}=divisor(\sigma)=0.
$$
There are $16$ variables and $10$ equations, each of
them of $4$ terms. That makes it similar to $\OGr(5,10)$.
One checks that the two sets of equations
and the two varieties are identical, although $\Cr(2,4,2)$ is
only quasi-homogeneous under $G = \GL(2) \times \GL(4) \times \GL(2)$
with an open orbit with complement $V_{0}$ of codimension $2$. 
The relation between $\OGr(5,10)$ and $\Cr(2,4,2)$ seems to be
an intriguing sporadic phenomenon that has possibly not been noticed before.

\section{Acknowledgements}
 This work was mainly carried out during my  PhD studies at the University of Warwick, UK under the supervision of Professor Miles Reid, to whom I would like to express my deepest gratitude. 



\begin{thebibliography}{9}
\bibitem{RC}
M. Reid, A. Corti, Weighted Grassmannians, in: Algebraic geometry, de
Gruyter, Berlin, 2002, pp. 141-163.\\

\bibitem{HOC}
M. Hochster, Topics in the homological theory of modules over commu-
tative rings, Published for the Conference Board of the Mathematical
Sciences by the American Mathematical Society, Providence, R.I., 1975,
expository lectures from the CBMS Regional Conference held at the Uni-
versity of Nebraska, Lincoln, Neb., June 24-28, 1974, Conference Board
of the Mathematical Sciences Regional Conference Series in Mathematics,
No. 24.\\

\bibitem{KM}
A. R. Kustin, M. Miller, Algebra structures on minimal resolutions of
Gorenstein rings of embedding codimension four, Math. Z. 173 (2) (1980)
171-184.

\bibitem{PR}
S. A. Papadakis, Kustin-Miller unprojection with complexes, J. Algebraic
Geom. 13 (2) (2004) 249-268.\\

\bibitem{PR1}
S. A. Papadakis, M. Reid, Kustin-Miller unprojection without complexes,
J. Algebraic Geom. 13 (3) (2004) 563-577.\\
\end{thebibliography}
\end{document}